\documentclass[12pt,twoside,reqno]{amsart}
\usepackage[latin1]{inputenc}
\usepackage[italian,english]{babel}
\usepackage{amssymb,latexsym}
\usepackage{amsmath,amsfonts,amsthm}
\usepackage{paralist}
\usepackage{color}

\numberwithin{equation}{section}
\newcommand{\R}{\mathbb{R}}
\newcommand{\N}{\mathbb{N}}

\newcommand{\Z}{\mathbb{Z}}

\newtheorem{lemma}{Lemma}
\newtheorem{thm}{Theorem}

\newtheorem{dfn}{Definition} 
\theoremstyle{remark}
\newtheorem{rem}{Remark}

\begin{document}
\title{A Fractional Landesman-Lazer type problem set on $\mathbb{R}^{N}$}
\author[Vincenzo Ambrosio]{Vincenzo Ambrosio}

\address{%
Dipartimento di Matematica e Applicazioni\\ Universit\`a degli Studi "Federico II" di Napoli\\
via Cinthia, 80126 Napoli, Italy}

\email{vincenzo.ambrosio2@unina.it}


\date{}

\maketitle

% Abstract
\begin{abstract}
\noindent
By using the abstract version of Struwe's monotonicity-trick we prove the existence of a positive solution to the problem
\begin{align*}
\left\{
\begin{array}{ll}
(-\Delta)^{s} u+Ku = f(x,u) \mbox{ in } \mathbb{R}^{N} \\
u\in H^{s}(\mathbb{R}^{N}), K>0
\end{array}
\right.  
\end{align*} 
where $f(x,t):\mathbb{R}^{N}\times \mathbb{R} \rightarrow \mathbb{R}$ is a Caratheodory function, $1$-periodic in $x$ and does not satisfy the Ambrosetti-Rabinowitz condition.
\end{abstract}

% The article itself
\section{Introduction}
In this paper we are concerned with the existence of positive solutions of the following problem
\begin{align}\label{P}
\left\{
\begin{array}{ll}
(-\Delta)^{s} u+Ku = f(x,u) \mbox{ in } \mathbb{R}^{N} \\
u\in H^{s}(\mathbb{R}^{N}), K>0
\end{array}
\right.  
\end{align}
where $s\in (0,1)$, $N>2s$ and $f:\mathbb{R}^{N}\times \mathbb{R} \rightarrow \mathbb{R}$ is a  Caratheodory function satisfying the following hypotheses: 
\begin{compactenum}[(H1)]
\item  $f(x,t)=0$ for any $t<0$ and a.e. $x\in \mathbb{R}^{N}$, $f(\cdot,t)\in L^{\infty}(\mathbb{R}^{N})$ and $f(\cdot,t)$ is $1$-periodic in $x_{i}$, $i=1,\dots, N$;
\item There exists $p\in (2,\frac{2N}{N-2s})$ such that $\displaystyle{\lim_{t\rightarrow \infty} \frac{f(x,t)}{t^{p-1}}=0}$ uniformly in $x\in \mathbb{R}^{N}$;
\item $\displaystyle{\lim_{t\rightarrow 0} \frac{f(x,t)}{t}=0}$ uniformly in $x\in \mathbb{R}^{N}$;
\item There exists $a\in (0,\infty]$ such that $\displaystyle{\lim_{t\rightarrow \infty} \frac{f(x,t)}{t}=a}$ uniformly in $x\in \mathbb{R}^{N}$.
\end{compactenum}
Let $G:\mathbb{R}^{N}\times \mathbb{R}_{+} \rightarrow \mathbb{R}$ be a function defined by setting
$$
G(x,t)=\frac{1}{2}f(x,t)t-F(x,t) \mbox{ where } F(x,t)=\int_{0}^{t} f(x,\tau) d\tau.
$$
Then, we also assume
\begin{compactenum}[(J1)]
\item $G(x,t)\geq 0$ for any $t\geq 0$, a.e. $x\in \mathbb{R}^{N}$ and there is $\delta>0$ such that
$$
f(x,t)t^{-1}\geq K-\delta \Rightarrow G(x,t)\geq \delta
$$
\item There exists $D\in [1,\infty)$ such that, a.e. $x\in \mathbb{R}^{N}$
$$
G(x,\bar{t})\leq D\,G(x,t) \mbox{ for any } 0<\bar{t}\leq t.
$$
\end{compactenum} 
\noindent
Here $(-\Delta)^{s}$ denotes the fractional Laplacian defined through the Fourier transform in the following way
$$
\mathcal{F}(-\Delta)^{s}u(k)=|k|^{2s}\mathcal{F}u(k) \quad (k\in \mathbb{R}^{N})
$$
for any $u\in \mathcal{S}(\mathbb{R}^{N})$.
It can be also computed by the following singular integral
$$
(-\Delta)^{s}u(x)=c_{N,s} \,P.V. \int_{\mathbb{R}^{N}} \frac{u(x)-u(y)}{|x-y|^{N+2s}} dy   \quad (x\in \mathbb{R}^{N}),
$$
where P.V. stands for the Cauchy principal value and $c_{N,s}$ is a normalization constant; see \cite{DPV} for more details.

\smallskip

\noindent
Equation (\ref{P}) appears in the study of the fractional Schr\"odinger equation
\begin{equation}\label{FSE}
\imath \frac{\partial \psi}{\partial t}+(-\Delta)^{s}\psi+K\psi=F(x,\psi) \mbox{ in } \mathbb{R}^{N}
\end{equation}
when looking for standing waves solutions that have the form $\psi(x,t)=e^{\imath \omega t}u(x)$ where $\omega\in \mathbb{R}$ and $u\in H^{s}(\mathbb{R}^{N})$.
This equation plays an important role in the study of the fractional quantum mechanics; see for example \cite{CW, DPPV, FQT, FL, Laskin3, LY1, Secchi}. For the interested reader we also suggest the Appendix A of \cite{DDDV} where a detailed physical description of (\ref{FSE}) is given. \\
%in which is given a detailed physical description of (\ref{FSE}). \\
When $s=1$, (\ref{FSE}) reduces to the classical nonlinear Schr\"odinger equation
\begin{equation}\label{SE}
-\Delta u+Ku = f(x,u) \mbox{ in } \mathbb{R}^{N}
\end{equation}
which has been extensively investigated in these last twenty years.
In the celebrated paper \cite{BL1}, Berestycki and Lions proved the existence of ground states to (\ref{SE}) (and a multiplicity result in \cite{BL2}), when $f$ is autonomous and satisfies the assumptions $(H1)-(H4)$. They work in the radially symmetric Sobolev space $H^{1}_{r}(\R^{N})$ of $H^{1}(\R^{N})$ and use a Lagrange multiplier procedure which is essentially based on the Pohozaev's Identity \cite{Poh} for (\ref{SE}).
%The loss of compactness in $\R^{N}$  was retained by working in radially symmetric Sobolev spaces.
When  $f$ is not autonomous, Pohozaev's identity provides no informations, so in many works concerning (\ref{SE}), it is usually assumed that $f(x,u)$ satisfies the Ambrosetti-Rabinowitz condition \cite{AR}, i.e.
\begin{equation}
\exists \mu>2, R>0 : 0< \mu F(x,t) \leq f(x,t)t, \, \forall |t|\geq R, \mbox{ a.e. } x\in \mathbb{R}^{N}.   \tag{AR}
\end{equation}
This condition is very crucial in applying the critical point theory, because, roughly speaking, it ensures the boundedness of the Palais-Smale sequences of the energy functional
%In this way it is possible to prove the existence of a positive solution to (\ref{SE}) via Mountain Pass Theorem because of, roughly speaking, (AR)-condition ensures the boundedness of the Palais-Smale sequences of the energy functional
$$
J(u)=\int_{\mathbb{R}^{N}} \frac{1}{2}[|\nabla u|^{2}+Ku^{2}] dx-\int_{\mathbb{R}^{N}} F(x,u) dx \quad u\in H^{1}(\mathbb{R}^{N})
$$ 
associated to the problem (\ref{SE}).
However, although (AR) is a quite natural condition, it is somewhat restrictive and eliminates many nonlinearities.
In fact, (AR) implies that for some $A,B>0$,
\begin{equation}\label{cAR}
F(x,t) \geq A |t|^{\mu} -B \mbox{ for any } t\in \mathbb{R}.
\end{equation}
Hence, for example,  the function 
\begin{equation}\label{examplef}
f(x,t)=t\log(1+|t|),
\end{equation}
does not satisfy the (AR)-condition.
For this reason, many authors studied (\ref{SE}) trying to drop the condition (AR).
One of the first result in this direction was due to Jeanjean \cite{J}.
To overcome the difficulty that the Palais-Smale sequences of $J$ may be unbounded, he developed an abstract version of  the monotonicity trick due to Struwe \cite{AS, Struwe} for functionals depending on a real parameter. \\
Here we recall his result:
\begin{thm}\cite{J}\label{absJ}
Let $(X, ||\cdot||)$ be a Banach space and $\Lambda\subset \mathbb{R}_{+}$ an interval. We consider a family $\{ I_{\lambda}\}_{\lambda\in \Lambda}$ of $C^{1}$-functionals on $X$ of the form
$$
I_{\lambda}(u)=A(u)-\lambda B(u), \quad \forall \lambda\in \Lambda
$$
where $B(u)\geq 0$ $\forall u\in X$ and such that either $A(u)\rightarrow +\infty$ or $B(u)\rightarrow +\infty$ as $||u||\rightarrow \infty$. \\
If there are two points $v_{1},v_{2}\in X$ such that 
$$
c_{\lambda}=\inf_{\gamma \in \Gamma} \max_{t\in [0,1[}I_{\lambda}(\gamma(t))>\max\{I_{\lambda}(v_{1}), I_{\lambda}(v_{2})\} \quad \lambda\in \Lambda,
$$
where
$$
\Gamma=\{\gamma\in C([0,1],X): \gamma(0)=v_{1},  \gamma(1)=v_{2}  \}, 
$$
then, for almost every $\lambda\in \Lambda$, there exists a sequence $\{v_{n}\}\subset X$ such that
\begin{compactenum}[(i)]
\item $\{v_{n}\}$ is bounded;
\item $I_{\lambda}(v_{n})\rightarrow c_{\lambda}$;
\item $I'_{\lambda}(v_{n})\rightarrow 0$ in the dual $X^{-1}$ of $X$.
\end{compactenum}
\end{thm}
This principle says, essentially, that given a family of $C^{1}$ functionals $I_{\lambda}$ satisfying a uniform Mountain Pass geometry and monotonically depending on the parameter $\lambda$, then the almost everywhere differentiability of the Mountain Pass value $c_{\lambda}$ induces the existence of a bounded Palais-Smale sequence for $I_{\lambda}$ for almost every $\lambda$ in the interval $\Lambda$ where the family is defined. \\
As application of Theorem \ref{absJ}, Jeanjean obtained the following existence result for the problem (\ref{SE}):
\begin{thm}\cite{J}\label{thmJ}
 Let $N>2$. Assume that $f$ satisfies the assumptions $(H1)-(H4)$ with $p\in (2, \frac{2N}{N-2})$. If $(J1)$ holds with $a<\infty$ in $(H4)$, then, if $K\in (0,a)$, there exists a non-trivial positive solution of (\ref{SE}).
If $(J2)$ holds with $a=\infty$ in $(H4)$, then there exists a non-trivial positive solution of (\ref{SE}).
\end{thm}

Let us notice that in the above Theorem \ref{thmJ}, the condition (AR) is replaced by $(J1)$ if $a<\infty$ or by $(J2)$ if $a=\infty$.
%By a direct integration of (AR), it easy to see that
%$$
%F(x,t) \geq A t^{\mu} -B \mbox{ for any } t\geq 0.
%$$
%for some $A,B>0$.
In fact, taking into account (\ref{cAR}), we can see that when $a<\infty$, (AR) does not hold, while if $a= \infty$ it may happen that (AR) is satisfied but, by using the assumptions on $f$, this is not possible. For example (AR) is not true for the function in (\ref{examplef}), which satisfies $(H1)-(H4)$ and $(J2)$.

%above existence result in \cite{J}
In this paper we claim to extend the above Theorem \ref{thmJ} for the nonlocal analogue of problem (\ref{SE}), by replacing the standard Laplacian operator by the fractional Laplacian operator. \\  
%We recall that in the recent years a great attention has been devoted to the study of non-local equations, in particular to the ones driven by the fractional Laplace operator. In fact such operator arises in several fields such as optimization, finance, phase transitions, stratified materials, anomalous diffusion, crystal dislocation, flame propagation, conservation laws, ultra-relativistic limits of quantum mechanics, quasi-geostrophic flows, minimal surfaces and water waves; see for instance \cite{bertoin, BKW, CSM, CSS, Cval, CV, CT, FJLL, SV} and references therein.
Recently, a great attention has been devoted to the study of non-local equations, in particular to the ones driven by the fractional Laplace operator. In fact such operator arises in several fields such as optimization, finance, phase transitions, stratified materials, anomalous diffusion, crystal dislocation, flame propagation, conservation laws, ultra-relativistic limits of quantum mechanics, quasi-geostrophic flows, minimal surfaces and water waves.
%; see for instance \cite{bertoin, BKW, CSM, CSS, Cval, CV, CT, FJLL, SV} and references therein. 
The literature is too wide to attempt a reasonable list of references here, so we derive the reader to the work by Di Nezza et al. \cite{DPV}, where a more extensive bibliography and an introduction to the subject are given.
We would just cite the papers "Mountain pass solutions for non-local elliptic operators. J. Math. Anal. Appl. 389 (2012), no. 2, 887-898" \cite{SV1} and "Variational methods for non-local operators of elliptic type. Discrete Contin. Dyn. Syst. 33 (2013), no. 5, 2105-2137" \cite{SV2} by R. Servadei \& E. Valdinoci, which are probably the first results dealing with nonlinear analysis in fractional setting. 

\noindent
Now we state our main result.
%In this framework our main result is the following one
\begin{thm}\label{thm1}
Let $s\in (0,1)$ and $N> 2s$. Assume that $(H1)-(H4)$ and $(J1)$ hold with $a<\infty$ in $(H4)$. Then if $K\in (0,a)$ there exists a non-trivial positive solution of (\ref{P}).
Assume that $(H1)-(H4)$ and $(J2)$ hold with $a=\infty$ in $(H4)$. Then there exists a non-trivial positive solution $u$ of (\ref{P}).
\end{thm}
\begin{rem}
By using similar arguments to those developed in \cite{FQT}, it is possible to prove that $u\in C^{0,\alpha}(\mathbb{R}^{N})$ for some $\alpha \in (0,1)$ and $u(x) \rightarrow 0$ as $|x| \rightarrow \infty$.
\end{rem}

\noindent
To prove Theorem \ref{thm1}, we follow the approach developed in \cite{J}. Several modifications will be necessary to deal with the non-local features of problem (\ref{P}). 
 %and we make the suitable modifications due to the presence of the fractional Laplacian.

We consider the following family of functionals 
$$
I_{\lambda}(u)=\frac{1}{2}||u||^{2}_{H^{s}(\mathbb{R}^{N})}-\lambda\int_{\mathbb{R}^{N}} F(x,u) dx
$$
with $\lambda\in [1,2]$, and we show that it satisfies the assumptions of Theorem \ref{absJ}. Then, for almost every $\lambda \in [1,2]$, there exists a bounded sequence $\{v_{m}\}\subset H^{s}(\mathbb{R}^{N})$ such that 
$$
I_{\lambda}(v_{m})\rightarrow c_{\lambda} \, \mbox{ and } \, I'_{\lambda}(v_{m}) \rightarrow 0 \mbox{ in } H^{-s}(\mathbb{R}^{N}). 
$$

By using the translational invariance of (\ref{P}) we obtain the existence of a sequence $\{y_{m}\}\subset \mathbb{Z}^{N}$ such that $u_{m}(x):=v(x-y_{m})\rightharpoonup u_{\lambda}\neq 0$ in $H^{s}(\mathbb{R}^{N})$, $I_{\lambda}(u_{\lambda})\leq c_{\lambda}$ and $I'_{\lambda}(u_{\lambda})=0$. 
By the weak maximum principle \cite{CS1} we have $u_{\lambda}\geq 0$ a.e. in $\mathbb{R}^{N}$. 
As a consequence we deduce the existence of a decreasing sequence $\{\lambda_{n}\}\subset [1,2]$ such that $\lambda_{n} \rightarrow 1$ and a sequence $\{u_{n}\}\subset H^{s}(\mathbb{R}^{N})$ such that $u_{n}\neq 0$, $I_{\lambda_{n}}(u_{n})\leq c_{\lambda_{n}}$ and $I'_{\lambda_{n}} (u_{n})=0$. 
We prove that $\{u_{n}\}$ is bounded and we show how this information to allow us to obtain a positive solution $u$ to (\ref{P}). 

\noindent
The paper is organized as follows: in Section $2$ we give a quick review about the fractional Sobolev spaces; in Section $3$ we give the proof of Theorem \ref{thm1}.

\section{Preliminaries}
In this section we collect some preliminaries facts about the fractional Sobolev spaces.
Let $s\in (0,1)$. We define the fractional Sobolev space by setting
$$
H^{s}(\mathbb{R}^{N})=\Bigl\{u\in L^{2}(\mathbb{R}^{N}): [u]_{H^{s}(\mathbb{R}^{N})}:= \sqrt{\iint_{\mathbb{R}^{2N}} \frac{|u(x)-u(y)|^{2}}{|x-y|^{N+2s}} dxdy}<\infty\Bigr\}
$$
which is a Hilbert space endowed with the norm
$$
||u||_{H^{s}(\mathbb{R}^{N})}=\sqrt{||u||^{2}_{L^{2}(\mathbb{R}^{N})}+[u]^{2}_{H^{s}(\mathbb{R}^{N})}}.
$$ 
By using the Plancherel's Theorem we can see \cite{DPV} that 
$$
[u]^{2}_{H^{s}(\mathbb{R}^{N})}=2C(N,s)^{-1} \int_{\mathbb{R}^{N}} |k|^{2s} |\mathcal{F}u(k)|^{2} dk.
$$
where 
$$
C(n,s):=\Bigl(\int_{\mathbb{R}^{N}} \frac{1-\cos x_{1}}{|x|^{N+2s}} dx\Bigr)^{-1}.
$$
We recall the following embedding:
\begin{thm}\label{ce}\cite{DPV}
Let $s\in (0,1)$ and $N>2s$. Then $H^{s}(\mathbb{R}^{N})$ is continuously embedded in $L^{q}(\mathbb{R}^{N})$ for any $q\in [2,\frac{2N}{N-2s}]$ and compactly embedded in $L^{q}_{loc}(\mathbb{R}^{N})$ for any $q\in [2,\frac{2N}{N-2s})$.
\end{thm}
Now we state the following results which we will use later
\begin{lemma}\label{br}\cite{Bre}
Let $\{u_{n}\}$ be a sequence in $L^{q}(\mathbb{R}^{N})$ with $q\in [1,\infty]$ and let $u\in L^{q}(\mathbb{R}^{N})$ be such that $u_{n} \rightarrow u$ in $L^{p}(\mathbb{R}^{N})$. Then there exists a subsequence $\{u_{n_{k}}\}$ and a function $h\in L^{q}(\mathbb{R}^{N})$ such that
\begin{itemize}
\item $ u_{n_{k}} \rightarrow u$ a.e. in $\mathbb{R}^{N}$;
\item $|u_{n_{k}}(x)|\leq h(x) $ a.e. $x\in \mathbb{R}^{N}$ and for any $k\in \mathbb{N}$.
\end{itemize}
\end{lemma}

\begin{lemma}\label{fqt}\cite{FQT}
Let $N\geq 2$. Assume that $\{u_{n}\}$ is bounded in $H^{s}(\mathbb{R}^{N})$ and it satisfies
$$
\lim_{n\rightarrow \infty} \sup_{y\in \mathbb{R}^{N}}\int_{B_{R}(y)} |u_{n}(x)|^{2} dx=0 
$$
where $R>0$. Then $u_{n} \rightarrow 0$ in $L^{q}(\mathbb{R}^{N})$ for $q\in (2, \frac{2N}{N-2s})$.
\end{lemma}

We conclude this section giving some estimates for the nonlinear term $f$ and its primitive $F$. This part is quite standard and the proof of the following Lemma can be found in \cite{Rab}.
\begin{lemma}
Assume $f:\mathbb{R}^{N}\times \mathbb{R} \rightarrow \mathbb{R}$ is a function satisfying conditions $(H1)-(H3)$.
Then for any $\varepsilon>0$ there exists $C_{\varepsilon}>0$ such that
\begin{equation}\label{(3.1)}
f(x,t)\leq \varepsilon t+C_{\varepsilon}t^{p-1} \mbox{ for } t\geq 0
\end{equation}
and so, as a consequence
\begin{equation}\label{(3.2)}
F(x,t)\leq \frac{\varepsilon}{2} t^{2}+\frac{C_{\varepsilon}}{p}t^{p} \mbox{ for } t\geq 0.
\end{equation}

\end{lemma}

\section{Positive solution of (\ref{P})}
In this section we give the proof of Theorem \ref{thm1}.
Firstly we recall the definition of weak solution to (\ref{P}).
\begin{dfn}
We say that $u\in H^{s}(\mathbb{R}^{N})$ is a weak solution to (\ref{P}) if 
\begin{equation}
\iint_{\mathbb{R}^{2N}}\frac{(u(x)-u(y))}{|x-y|^{N+2s}}(\varphi(x)-\varphi(y))dxdy+\int_{\mathbb{R}^{N}}Ku\varphi dx=\int_{\mathbb{R}^{N}}f(x,u)\varphi dx
\end{equation}
for any $\varphi \in H^{s}(\mathbb{R}^{N})$. 
\end{dfn}

Let us consider the functional 
\begin{equation*}\label{functional}
I(u)=\frac{1}{2}||u||^{2}_{H^{s}(\mathbb{R}^{N})}-\int_{\mathbb{R}^{N}} F(x,u) dx
\end{equation*}
for $u\in H^{s}(\mathbb{R}^{N})$.  Here we use the notation
$$
||u||^{2}_{H^{s}(\mathbb{R}^{N})}:=[u]^{2}_{H^{s}(\mathbb{R}^{N})}+K||u||^{2}_{L^{2}(\mathbb{R}^{N})}
$$
which is equivalent to the standard norm in $H^{s}(\mathbb{R}^{N})$ (defined in Section $2$) since $K>0$.
Then it is clear that $I$ is well defined, $I(0)=0$, $I\in C^{1}(H^{s}(\mathbb{R}^{N}), \mathbb{R})$ and the critical points of $I$ are weak solutions to (\ref{P}).

We begin proving that $I$ has a Mountain-Pass geometry:
\begin{lemma}\label{lemma3.1}
Assume that $(H1)-(H3)$ hold. Then 
$$
I(u)=\frac{1}{2}||u||^{2}_{H^{s}(\mathbb{R}^{N})}+o(||u||^{2}_{H^{s}(\mathbb{R}^{N})})\quad \mbox{ as } ||u||_{H^{s}(\mathbb{R}^{N})}\rightarrow 0.
$$
\end{lemma}
\begin{proof}
By using (\ref{(3.2)}) and Theorem \ref{ce} we get
\begin{equation*}
\int_{\R^{N}}F(x,u) dx \leq \frac{\varepsilon}{2} ||u||^{2}_{H^{s}(\R^{N})}+\frac{C'_{\varepsilon}}{p} ||u||^{p}_{H^{s}(\R^{N})}
\end{equation*}
which implies that $\int_{\R^{N}} F(x,u) \,dx=o(||u||^{2}_{H^{s}(\R^{N})})$ as $||u||_{H^{s}(\R^{N})}\rightarrow 0$. 

\end{proof}

\begin{lemma}\label{lemma3.2}
Assume that $(H1), (H2), (H4)$ hold and that $K\in (0, a)$. Then we can find a function $v\in H^{s}(\R^{N})$ such that $v\neq 0$ and $I(v)\leq 0$. 
\end{lemma}

\begin{proof}
For simplicity we assume $a<\infty$. 

Let us introduce
$$
d^{2}(N):= \int_{\R^{N}} e^{-2|x|^{2}} dx \mbox{ and } D(N):= \frac{2}{C(s,N)}d(N)^{-2} \int_{\R^{N}} |k|^{2s} e^{-2|k|^{2}} dk.
$$
For $\alpha>0$ we set $\displaystyle{w_{\alpha}(x)= d(N)^{-1} \alpha^{\frac{N}{4}} e^{-\alpha |x|^{2}}}$. 
Then it is easy to prove that 
\begin{align*}
w_{\alpha}\in H^{s}(\R^{N}),\,
||w_{\alpha}||_{L^{2}(\R^{N})}=1 \, \mbox{ and } \,[w_{\alpha}]_{H^{s}(\R^{N})}^{2} = \alpha^{s} D(N). 
\end{align*}
Fix $\alpha \in (0, \Bigl(\frac{a-K}{D(N)}\Bigr)^{\frac{1}{s}})$. Thus we deduce that 
\begin{equation}\label{(3.3)}
[w_{\alpha}]_{H^{s}(\R^{N})}^{2}< a-K. 
\end{equation}
Since $t w_{\alpha} (x) \rightarrow +\infty$ as $t\rightarrow \infty$, by $(H4)$ we have
\begin{equation}\label{t1}
\lim_{t\rightarrow +\infty} \frac{F(x,tw_{\alpha})}{t^{2} w_{\alpha}^{2}} = \frac{a}{2} \,\mbox{ a.e. } x\in \R^{N}. 
\end{equation}
On the other hand, by using $(H1), (H3)$, and $(H4)$ we obtain the existence of a positive constant $C$ such that
%Let us note that $(H1), (H3), (H4)$ imply the existence of a positive constant $C$ such that
\begin{equation}\label{t2}
0\leq \frac{F(x,t)}{t^{2}}\leq C \, \mbox{ for any } t\in \R \mbox{ and a.e. } x\in \R^{N}.
\end{equation}
Then, taking into account (\ref{t1}) and (\ref{t2}), and by using the Dominated Convergence Theorem we can see 
\begin{equation}\label{t3}
\lim_{t\rightarrow +\infty} \int_{\R^{N}} \frac{F(x,tw_{\alpha})}{t^{2}} = \frac{a}{2} \int_{\R^{N}} w_{\alpha}^{2} \, dx = \frac{a}{2}.
\end{equation}
As a consequence, by using (\ref{(3.3)}) and (\ref{t3}) we obtain 
\begin{align*}
\lim_{t\rightarrow +\infty} \frac{I(tw_{\alpha})}{t^{2}} &= \frac{1}{2} [w_{\alpha}]^{2}_{H^{s}(\R^{N})} + \frac{K}{2} || w_{\alpha}||^{2}_{L^{2}(\R^{N})} - \lim_{t\rightarrow +\infty} \int_{\R^{N}} \frac{F(x, tw_{\alpha})}{t^{2}} dx \\
&=\frac{1}{2} [w_{\alpha}]^{2}_{H^{s}(\R^{N})} + \frac{K}{2} -\frac{a}{2} <0.
\end{align*}
\end{proof}

To construct a solution of (\ref{P}), we introduce the following parametrized family of functionals
\begin{equation}\label{ilambda}
I_{\lambda}(u)=\frac{1}{2}||u||^{2}_{H^{s}(\R^{N})}-\lambda\int_{\R^{N}} F(x,u) dx \mbox{ with } \lambda\in [1,2].
\end{equation}
%Since $I_{1}=I$, we will construct bounded Palais-Smale sequences for a.e. $\lambda$ close to $1$ and we exploit the Theorem \ref{absJ}. 

%we exploit the following result
%\begin{thm}\cite{J}
%Let $(X, ||\cdot||)$ be a Banach space and $J\subset \R_{+}$ be an interval. We consider a family $(I_{\lambda})_{\lambda \in J}$ of $C^{1}$-functionals of $X$ of the form
%$$
%I_{\lambda}(u)=A(u)-\lambda B(u)
%$$
%where $B(u)\geq 0$ for any $u\in X$ and such that either $A(u)\rightarrow +\infty$ or $B(u)\rightarrow +\infty$ as $||u||\rightarrow +\infty$. We assume that there are two points $v_{1}$ and $v_{2}$ in $X$ such that setting
%$$
%\Gamma = \{ \gamma \in C([0,1], X) : \gamma(0)=v_{1} \mbox{ and } \gamma(1)=v_{2}\}. 
%$$
%there hold, for any $\lambda \in J$
%$$
%c_{\lambda} := \inf_{\gamma \in \Gamma} \max_{t\in [0,1]} I_{\lambda} (\gamma(t))>\max\{I_{\lambda}(v_{1}), I_{\lambda}(v_{2})\}.
%$$
%Then for almost every $\lambda \in J$, there exists a sequence $\{v_{n}\}\subset X$ such that 
%\begin{itemize}
%\item $\{v_{n}\}$ is bounded; 
%\item $I_{\lambda_{n}}(v_{n})\rightarrow c_{\lambda}$;
%\item $I'_{\lambda_{n}}(v_{n})\rightarrow 0$ in the dual $X^{-1}$ of $X$.
%\end{itemize}
%\end{thm}

Thus, we are ready to prove 
\begin{lemma}\label{lemma3.3}
Assume that $(H1)-(H4)$ hold. Then the family $(I_{\lambda})_{\lambda \in [1,2]}$ defined in (\ref{ilambda}) satisfies the hypotheses of Theorem \ref{absJ}. In particular for almost every $\lambda\in [1,2]$ there exists a bounded sequence $\{v_{m}\} \subset H^{s}(\R^{N})$ such that
$$
I_{\lambda} (v_{m}) \rightarrow c_{\lambda}\, \mbox{ and } \, I'_{\lambda}(v_{m}) \rightarrow 0 \, \mbox{ in } \, H^{-s}(\R^{N}). 
$$
\end{lemma}

\begin{proof}
Let $v\in H^{s}(\R^{N})$ be the function obtained in Lemma \ref{lemma3.2}. Then we have $I_{\lambda} (v) \leq 0$ for all $\lambda \geq 1$ since
$$
\int_{\R^{N}} F(x,u) \, dx \geq 0, \quad \forall u\in H^{s}(\R^{N}). 
$$ 
By Lemma \ref{lemma3.1} follows that
$$
\int_{\R^{N}} F(x,u) \, dx =o(||u||^{2}_{H^{s}(\R^{N})}) \,\mbox{ as }\, ||u||_{H^{s}(\R^{N})}\rightarrow 0. 
$$
Then, for any $\lambda \in [1,2]$ we have 
$$
c_{\lambda} = \inf_{\gamma \in \Gamma} \max_{t\in [0,1]} I_{\lambda} (\gamma(t)) >0
$$
where 
$$
\Gamma = \{ \gamma \in C([0,1], H^{s}(\R^{N})) : \gamma(0)=0 \mbox{ and } \gamma(1)=v\}. 
$$
Therefore, we are in the position to apply Theorem \ref{absJ}.  

\end{proof}

Now we give the following terminology which we will often use later. Let $\{ u_{n} \}\subset H^{s}(\R^{N})$ be an arbitrary sequence. We say that $\{ u_{n} \}$ does not vanish if it is possible to translate each $u_{n}$ so that the translated sequence (still denoted by $\{u_{n} \}$) satisfies, up to a subsequence, the following condition:
there exists $\alpha>0$ and $R<\infty$ such that 
$$
\lim_{n\rightarrow \infty} \int_{B_{R}} u_{n}^{2} dx\geq \alpha>0.
$$
If it is not the case then necessarily one has 
$$
\lim_{n\rightarrow \infty} \sup_{y\in \mathbb{Z}^{N}}\int_{y+B_{R}} u_{n}^{2} dx=0 \mbox{ for any } R<\infty
$$
and we say that $\{ u_{n} \}$ vanishes.

\begin{lemma} \label{lemma3.4}
Assume that $(H1)-(H3)$ hold. Let $\{u_{n}\}\subset H^{s}(\R^{N})$ be a bounded sequence which vanishes. Then 
$$
\lim_{n\rightarrow +\infty} \int_{\R^{N}} G(x, u_{n}) \, dx =0. 
$$ 
\end{lemma}

\begin{proof}
By using Lemma \ref{fqt} we know that 
\begin{equation}\label{t5}
u_{n}\rightarrow 0 \mbox{ in } L^{q}(\R^{N}) \mbox{ for any } q\in \Bigl(2, \frac{2N}{N-2s}\Bigr). 
\end{equation}
Taking into account (\ref{(3.1)}), (\ref{(3.2)}), Theorem \ref{ce} and (\ref{t5}), and by using the fact that $u_{n}$ is bounded, we can see that
\begin{align*}
\int_{\R^{N}} G(x, u_{n}) \, dx = \int_{\R^{N}} \Bigl[ \frac{1}{2} f(x, u_{n}) u_{n} - F(x, u_{n})\Bigr] \, dx \rightarrow 0 \, \mbox{ as } n\rightarrow \infty. 
\end{align*}
\end{proof}

Now we prove the following result
\begin{lemma}\label{lemma3.5}
Assume that $(H1)-(H4)$ and either $(J1)$ or $(J2)$ hold. Let $\lambda \in [1,2]$ be fixed. Let $\{v_{m}\}\subset H^{s}(\R^{N})$ be a  bounded sequence such that 
\begin{compactenum}[(I)]
\item $\displaystyle{0<\lim_{m\rightarrow +\infty} I_{\lambda} (v_{m}) \leq c_{\lambda}}$;
\item $I'_{\lambda}(v_{m}) \rightarrow 0 \mbox{ in } H^{-s}(\R^{N})$. 
\end{compactenum}
Then there exists $\{y_{m}\}\subset \mathbb{Z}^{N}$ such that, up to a subsequence, $u_{m}(x):= v_{m}(x-y_{m})$ satisfies 
\begin{compactenum}[(i)]
\item $u_{m}\rightharpoonup u_{\lambda} \neq 0$;
\item $I_{\lambda} (u_{\lambda}) \leq c_{\lambda}$; 
\item $I'_{\lambda}(u_{\lambda}) =0$. 
\end{compactenum}
\end{lemma}

\begin{proof}
Taking into account $(I),(II)$ and the boundedness of $v_{m}$ we have 
$$
\lim_{m\rightarrow \infty} \int_{\R^{N}} G(x, v_{m}) \, dx = \lim_{m\rightarrow \infty} [I_{\lambda} (v_{m}) -\frac{1}{2} I'_{\lambda}(v_{m})v_{m} ] >0. 
$$
Then, by Lemma \ref{lemma3.4} we can see that $v_{m}$ does not vanish, so there exists $\{y_{m}\}\subset \Z^{N}$ such that, up to a subsequence, $u_{m}(x) = v_{m}(x- y_{m})$ satisfies the following condition: there exist $\alpha>0$ and  $R<\infty$ such that
\begin{equation}\label{(3.5)}
\lim_{m\rightarrow \infty} \int_{B_{R}} u_{m}^{2} dx \geq \alpha >0.  
\end{equation}
Since the problem (\ref{P}) is invariant under the translation group associated to the periodicity of $f(\cdot, t)$, we have 
\begin{compactenum}[(a)]
\item $0<\lim_{m\rightarrow +\infty} I_{\lambda} (u_{m}) \leq c_{\lambda}$;
\item $I'_{\lambda}(u_{m}) \rightarrow 0 \mbox{ in } H^{-s}(\R^{N})$;
\item $u_{m} \rightharpoonup u_{\lambda}$, for some $u_{\lambda} \in H^{s}(\R^{N})$. 
\end{compactenum}
Then, $(i)$ follows by (c), (\ref{(3.5)}) and Theorem \ref{ce}. 
In order to prove $(iii)$, it is enough to show that $I'_{\lambda}(v) \varphi =0$ for all $\varphi \in C^{\infty}_{0}(\R^{N})$, since $C^{\infty}_{0}(\R^{N})$ is dense in $H^{s}(\R^{N})$ (see \cite{DPV}). %Since $C^{\infty}_{0}(\R^{N})$ is dense in $H^{s}(\R^{N})$ (see \cite{DPV}), it is enough to check that $I'_{\lambda}(v) \varphi =0$ for all $\varphi \in C^{\infty}_{0}(\R^{N})$. 
Taking into account  $u_{m}\rightharpoonup u_{\lambda}$ in $H^{s}(\R^{N})$ and $u_{m}\rightarrow u_{\lambda}$ in $L^{q}_{loc}(\R^{N})$ for any $q\in [2, \frac{2N}{N-2s})$, we get
$$
I'_{\lambda}(u_{m}) \varphi - I'_{\lambda}(u_{\lambda})\varphi = (u_{m}- u_{\lambda}, \varphi )_{H^{s}(\R^{N})} - \int_{\R^{N}} (f(x, u_{m}) - f(x, u_{\lambda}) ) \varphi \, dx \rightarrow 0.
$$ 
Then $(iii)$  follows by (b). Finally we verify $(iv)$. We note that either $(J1)$ or $(J2)$ imply that 
\begin{equation*}
%\label{(3.6)}
G(x,t)\geq 0 \mbox{ for all } t\geq 0 \mbox{ and a.e. } x\in \R^{N}.
\end{equation*}
So, by using Fatou's Lemma we can see that 
\begin{align*}
c_{\lambda} &\geq \lim_{m\rightarrow +\infty} \Bigl[I_{\lambda}(u_{m}) - \frac{1}{2} I'_{\lambda}(u_{m}) u_{m}\Bigr] \\
&= \lim_{m\rightarrow +\infty} \int_{\R^{N}} G(x, u_{m}) \, dx \\
&\geq \int_{\R^{N}} G(x, u_{\lambda}) \, dx \\
&= I_{\lambda}(u_{\lambda}) - \frac{1}{2} I'_{\lambda}(u_{\lambda}) u_{\lambda} =  I_{\lambda}(u_{\lambda}).
\end{align*}
\end{proof}

Combining Lemma \ref{lemma3.3} and Lemma \ref{lemma3.5} we obtain the existence of two sequences $\{\lambda_{n}\} \subset [1,2]$ and $\{u_{n}\} \subset H^{s}(\R^{N})$ such that 
\begin{itemize}
\item $\lambda_{n}\rightarrow 1$ and $\{\lambda_{n}\}$ is decreasing;
\item $u_{n} \neq 0$, $I_{\lambda_{n}}(u_{n})\leq c_{\lambda_{n}}$ and $I'_{\lambda_{n}}(u_{n})=0$. 
\end{itemize}
Let us observe that $u_{n}\geq 0$ a.e. in $\R^{N}$ (it is enough to multiply $(-\Delta)^{s}u_{n}+Ku_{n}=\lambda_{n} f(x,u_{n})$ in $\R^{N}$ by the negative part of $u_{n}$ and then one uses the assumption $(H1)$).

Taking into account
$$
\frac{1}{2} ||u_{n}||^{2}_{H^{s}(\R^{N})} - \lambda_{n} \int_{\R^{N}} F(x, u_{n}) dx \leq c_{\lambda_{n}},
$$
$$
||u_{n}||^{2}_{H^{s}(\R^{N})} = \lambda_{n} \int_{\R^{N}} f(x, u_{n}) u_{n} dx 
$$
and the fact that $\frac{c_{\lambda_{n}}}{\lambda_{n}}$ is increasing, we deduce 
\begin{equation}\label{(3.7)}
\int_{\R^{N}} G(x, u_{n}) dx \leq \frac{c_{\lambda_{n}}}{\lambda_{n}} \leq c_{1} \quad  \forall n\in \N. 
\end{equation}

\begin{lemma}\label{lemma3.6}
Assume that $(H1)-(H4)$ and either $(J1)$ or $(J2)$ hold. If the sequence $\{u_{n}\}\subset H^{s}(\R^{N})$ given above is bounded, then there exists $u\in H^{s}(\R^{N})$, $u\neq 0$ such that $I'(u)=0$.  
\end{lemma}

\begin{proof}
Firstly we observe that for any $v\in H^{s}(\R^{N})$ 
$$
I'(u_{n})v = I'_{\lambda_{n}}(u_{n})v + (\lambda_{n} -1) \int_{\R^{N}} f(x, u_{n}) v \, dx \rightarrow 0 
$$
and 
$$
I(u_{n})= I_{\lambda_{n}}(u_{n}) + (\lambda_{n} -1) \int_{\R^{N}} F(x, u_{n})  \, dx.
$$
Now we distinguish two cases:\\
First case: $\limsup_{n\rightarrow \infty} I_{\lambda_{n}}(u_{n})>0$. Then $\limsup_{n\rightarrow \infty} I(u_{n})>0$ and the thesis follows by Lemma \ref{lemma3.5}. \\
Second case: $\limsup_{n\rightarrow \infty} I_{\lambda_{n}}(u_{n})\leq 0$. 

Let us consider $t_{n}\in [0,1]$ such that
\begin{equation}\label{(3.8)}
I_{\lambda_{n}}(t_{n} u_{n}) = \max_{t\in [0,1]} I_{\lambda_{n}}(tu_{n}).
\end{equation}
Let $z_{n}=t_{n}u_{n}$ and observe that $\{z_{n}\}$ is bounded in $H^{s}(\R^{N})$. Since $I'_{\lambda_{n}}(z_{n}) z_{n} =0$ for any $n\in \N$, we have 
\begin{equation}\label{(3.9)}
\lambda_{n} \int_{\R^{N}} G(x, z_{n}) \, dx = I_{\lambda_{n}}(z_{n}) -\frac{1}{2} I'_{\lambda_{n}}(z_{n})z_{n} = I_{\lambda_{n}}(z_{n}).
\end{equation}
Proceeding as in the proof of Lemma \ref{lemma3.1} we can see that $I'_{\lambda_{n}}(u)u= ||u||^{2}_{H^{s}(\R^{N})} + o(||u||^{2}_{H^{s}(\R^{N})})$ as $||u||_{H^{s}(\R^{N})}\rightarrow 0$, uniformly in $n\in \N$. Then, being $I'_{\lambda_{n}}(u_{n})=0$, there exists $\alpha>0$ such that $||u_{n}||_{H^{s}(\R^{N})}\geq \alpha$ for all $n\in \N$. 

Putting together $\limsup_{n\rightarrow \infty} I_{\lambda_{n}}(u_{n})\leq 0$, Lemma \ref{lemma3.1}, (\ref{(3.8)}), (\ref{(3.9)}) and $\lambda_{n}\rightarrow 1$ we have
$$
\liminf_{n\rightarrow \infty} \int_{\R^{N}} G(x, z_{n}) \, dx= \liminf_{n\rightarrow \infty} I_{\lambda_{n}}(z_{n}) >0. 
$$ 
Then, Lemma \ref{lemma3.4} implies that $z_{n}$ (so $u_{n}$) does not vanish. Proceeding as in the proof of Lemma \ref{lemma3.5} we obtain the assertion. 

\end{proof}

Then,  taking into account Lemma \ref{lemma3.6}, it is enough to prove that $\{u_{n}\}\subset H^{s}(\R^{N})$ is bounded to conclude the proof of Theorem \ref{thm1}.
\begin{proof}(end of proof of Theorem \ref{thm1})
We argue by contradiction and we assume that $||u_{n}||_{H^{s}(\R^{N})}\rightarrow \infty$.

Let us consider the sequence
$$
w_{n}=\frac{u_{n}}{||u_{n}||_{H^{s}(\R^{N})}}.
$$
Then $||w_{n}||_{H^{s}(\R^{N})}=1$ and we can assume that $w_{n} \rightharpoonup w$ in $H^{s}(\R^{N})$. As a consequence either $w_{n}$ vanishes or it does not vanish.
We will prove that none of these alternatives occur and this gives a contradiction.

\begin{itemize}
\item Step 1: $w_{n}$ does not vanish. 
\end{itemize}

Proceeding as in the proof of Lemma \ref{lemma3.5} and by using the translation invariance of problem (\ref{P}), we can assume that $w_{n} \rightharpoonup w\neq 0$ in $H^{s}(\R^{N})$ and $w_{n} \rightarrow w$ a.e. in $\R^{N}$. 
Now we distinguish two cases.

Firstly we assume that $a<\infty$ in $(H4)$ and $K\in (0,a)$.
We prove that $w\neq 0$ satisfies the eigenvalue problem
\begin{equation*}
(-\Delta)^{s}w+Kw=aw \mbox{ in } \R^{N}
\end{equation*}
that is, for any $\varphi\in C^{\infty}_{0}(\R^{N})$
\begin{equation}\label{(3.11)}
\iint_{\R^{2N}}\frac{(w(x)-w(y))}{|x-y|^{N+2s}}(\varphi(x)-\varphi(y))dxdy+\int_{\R^{N}}Kw\varphi dx=\int_{\R^{N}}aw\varphi dx.
\end{equation}
This gives a contradiction since $(-\Delta)^{s}$ has no eigenvalue in $H^{s}(\R^{N})$.
To see this last fact, we can observe that if $\mu \in \R$ and $u\in H^{s}(\R^{N})$ satisfies $(-\Delta)^{s} u = \mu u$ in $\R^{N}$, by using the Pohozaev identity proved in \cite{CW}, we can deduce that 
\begin{equation*}
\frac{\mu N}{2} \int_{\R^{N}} u^{2} dx = \frac{N-2s}{2} \int_{\R^{N}} |k|^{2s} |\mathcal{F}u(k)|^{2} dk = \mu \frac{N-2s}{2} \int_{\R^{N}} u^{2} dx,
\end{equation*}
which necessarily implies that $u\equiv 0$.  \\
Now, we are going to prove (\ref{(3.11)}). Since $I'_{\lambda_{n}}(u_{n})=0$ we can see that  $w_{n}$ satisfies 
\begin{align*}
%\label{(3.12)}
\iint_{\R^{2N}}\frac{(w_{n}(x)-w_{n}(y))}{|x-y|^{N+2s}}(\varphi(x)-\varphi(y))dxdy &+\int_{\R^{N}}Kw_{n}\varphi dx\\
&=\int_{\R^{N}}\lambda_{n}\frac{f(x,u_{n})}{u_{n}}w_{n}\varphi dx
\end{align*}
for any $\varphi\in C^{\infty}_{0}(\R^{N})$.
By using the fact that $w_{n} \rightharpoonup w$ in $H^{s}(\R^{N})$ we get
\begin{align*}
\iint_{\R^{2N}}&\frac{(w_{n}(x)-w_{n}(y))}{|x-y|^{N+2s}}(\varphi(x)-\varphi(y))dxdy+\int_{\R^{N}}Kw_{n}\varphi dx\\
&\rightarrow \iint_{\R^{2N}}\frac{(w(x)-w(y))}{|x-y|^{N+2s}}(\varphi(x)-\varphi(y))dxdy+\int_{\R^{N}}Kw\varphi dx
\end{align*}
for any $\varphi\in C^{\infty}_{0}(\R^{N})$.

To obtain (\ref{(3.11)}) we have to prove that 
\begin{equation}\label{3.18}
\int_{\R^{N}}\lambda_{n}\frac{f(x,u_{n})}{u_{n}}w_{n}\varphi dx\rightarrow \int_{\R^{N}}aw\varphi dx
\end{equation}
Firstly we show 
\begin{equation}\label{(3.14)}
\lambda_{n}\frac{f(x,u_{n})}{u_{n}}w_{n}\rightarrow aw \mbox{ a.e. in } \R^{N}.
\end{equation}
We distinguish when $w(x)=0$ and $w(x)\neq 0$ (without loss of generality we can suppose that $w\neq 0$ is defined everywhere in $\R^{N}$).

Fix $x\in \R^{N}$ such that $w(x)=0$. By using $(H1),(H3)$ and $(H4)$ we can see that there exists $C<\infty$ such that
\begin{equation}\label{(3.15)}
0\leq \frac{f(x,t)}{t}\leq C\mbox{ for all } t\geq 0 \mbox{ a.e. in } \R^{N}.
\end{equation}
Since $\lambda_{n}$ is bounded and $w_{n}\rightarrow w$ a.e. in $\R^{N}$, we have for such $x\in \R^{N}$
$$
0\leq \lambda_{n}\frac{f(x,u_{n}(x))}{u_{n}(x)}w_{n}(x)\leq \lambda_{n}Cw_{n}(x)\rightarrow 0=aw(x).
$$
Now, let  $x\in \R^{N}$ be such that $w(x)\neq 0$. Then $u_{n}(x)\rightarrow \infty$ and by using $(H4)$ and $\lambda_{n} \rightarrow 1$ we have
$$
\lambda_{n}\frac{f(x,u_{n}(x))}{u_{n}(x)}w_{n}(x)\rightarrow aw(x).
$$
Therefore, we have proved (\ref{(3.14)}).
At this point, we fix $\varphi\in C^{\infty}_{0}(\R^{N})$ and let $\Omega$ be a compact set such that $supp(\varphi)\subset \Omega$. Since $H^{s}(\Omega)$ is compactly embedded in $L^{1}(\Omega)$ we have $w_{n} \rightarrow w$ in $L^{1}(\Omega)$. By Lemma \ref{br} we deduce the existence of a function $h\in L^{1}(\Omega)$ such that $w_{n}\leq h$ a.e. in $\Omega$, and by using (\ref{(3.15)}) we get
$$
0\leq \lambda_{n}\frac{f(x,u_{n})}{u_{n}}w_{n} \leq C w_{n} \leq C h \mbox{ a.e. } x\in \Omega. 
$$
This last fact, (\ref{(3.14)}) and the Dominated Convergence Theorem imply (\ref{3.18}).

Secondly we assume that $a=\infty$ in $(H4)$.
Since $u_{n}$ solves weakly
\begin{equation}\label{(3.211)}
(-\Delta)^{s}u_{n}+Ku_{n}=\lambda_{n}f(x,u_{n}) \mbox{ in } \R^{N}
\end{equation}
we deduce that $w_{n}$ satisfies 
\begin{align}\label{eqwns}
\iint_{\R^{2N}}\frac{(w_{n}(x)-w_{n}(y))}{|x-y|^{N+2s}}(\varphi(x)-\varphi(y))dxdy &+\int_{\R^{N}}Kw_{n}\varphi dx \nonumber \\
&=\int_{\R^{N}}\lambda_{n}\frac{f(x,u_{n})}{u_{n}}w_{n}\varphi dx
\end{align}
for any $\varphi\in H^{s}(\R^{N})$.
Then, being $w_{n} \rightharpoonup w$ in $H^{s}(\R^{N})$, we get
\begin{align}\label{(3.191)}
\lim_{n\rightarrow \infty} &\int_{\R^{N}} \frac{f(x,u_{n})}{||u_{n}||_{H^{s}(\R^{N})}}\varphi dx\nonumber\\ 
&=\iint_{\R^{2N}}\frac{(w(x)-w(y))}{|x-y|^{N+2s}}(\varphi(x)-\varphi(y))dxdy+\int_{\R^{N}}Kw\varphi dx.
\end{align}
Taking $\varphi=w$ in (\ref{(3.191)}) we deduce that 
\begin{equation}\label{(3.20)}
\lim_{n\rightarrow \infty} \int_{\R^{N}}\frac{f(x,u_{n})}{||u_{n}||_{H^{s}(\R^{N})}}w dx=||w||_{H^{s}(\R^{N})}^{2}.
\end{equation}
Now, let $\Omega=\{x\in \R^{N}:w(x)\neq 0\}$. Since $a=\infty$ we have 
$$
\frac{f(x,u_{n})}{||u_{n}||_{H^{s}(\R^{N})}}w=\frac{f(x,u_{n})}{u_{n}}w_{n}w \rightarrow +\infty \mbox{ a.e. in } \R^{N}.
$$
Taking into account that $|\Omega|>0$ and by using Fatou's Lemma we obtain
\begin{equation*}
+\infty\leq \liminf_{n\rightarrow \infty} \int_{\R^{N}}\frac{f(x,u_{n})}{||u_{n}||_{H^{s}(\R^{N})}}w dx=||w||_{H^{s}(\R^{N})}^{2}<\infty,
\end{equation*}
that is a contradiction.

\begin{itemize}
\item Step 2: $w_{n}$ vanishes. 
\end{itemize}

As in the Step $1$ we have to consider two cases.
Assume that $a<+\infty$ in $(H4)$ and $(J1)$ hold. Since $u_{n}$ solves (\ref{(3.211)}) we can see that $w_{n}$ satisfies (\ref{eqwns}). Taking $\varphi= w_{n} $ in (\ref{eqwns}), and recalling that $||w_{n}||_{H^{s}(\R^{N})}=1$, we get
\begin{equation}\label{(3.22)}
1=\lim_{n\rightarrow +\infty} \int_{\R^{N}} \frac{f(x, u_{n})}{u_{n}} w_{n}^{2} dx. 
\end{equation}
Set 
$$
\Omega_{n}=\Bigl \{x\in \R^{N} : \frac{f(x,u_{n})}{u_{n}}\leq K-\frac{\delta}{2} \Bigr\},
$$
where $\delta$ is defined as in $(J1)$. Since $1= ||w_{n}||_{H^{s}(\R^{N})} = [w_{n}]_{H^{s}(\R^{N})}^{2} + K ||w_{n}||_{L^{2}(\R^{N})}^{2}$ we can see that 
\begin{align*}
\int_{\Omega_{n}} \frac{f(x,u_{n})}{u_{n}} w_{n}^{2} dx &\leq \Bigl(K-\frac{\delta}{2}\Bigr)  \int_{\Omega_{n}} w_{n}^{2} dx \leq \frac{1}{K} \Bigl(K-\frac{\delta}{2}\Bigr)
\end{align*}
which together with (\ref{(3.22)}) imply 
\begin{equation}\label{(3.23)}
\liminf_{n\rightarrow +\infty} \int_{\R^{N}\setminus \Omega_{n}} \frac{f(x,u_{n})}{u_{n}} w_{n}^{2} dx\geq \frac{\delta}{2K} >0. 
\end{equation}
Now, we claim to prove that 
\begin{equation}\label{(3.24)}
\limsup_{n\rightarrow +\infty} |\R^{N}\setminus \Omega_{n}|= +\infty. 
\end{equation}
We argue by contradiction and we suppose that 
\begin{equation}\label{(3.25)}
\limsup_{n\rightarrow +\infty} |\R^{N}\setminus \Omega_{n}|< +\infty. 
\end{equation}
Taking into account (\ref{(3.15)}), (\ref{(3.23)}), (\ref{(3.25)}) and the fact that $w_{n}$ vanishes, we deduce that
\begin{align*}
0<\lim_{n\rightarrow +\infty} \int_{\R^{N}\setminus \Omega_{n}} \frac{f(x,u_{n})}{u_{n}} w_{n}^{2} dx \leq C \lim_{n\rightarrow +\infty} \int_{\R^{N}\setminus \Omega_{n}} w_{n}^{2} dx =0
\end{align*}
that is a contradiction.
Now, by using (\ref{(3.7)}) and the fact that $G(x, t)\geq 0$ for any $t\geq 0$ by $(J1)$, we have 
\begin{align*}
c_{1} &\geq \int_{\R^{N}} G(x, u_{n}) \, dx \\
&= \int_{\Omega_{n}} G(x, u_{n}) \, dx+  \int_{\R^{N}\setminus \Omega_{n}} G(x, u_{n}) \, dx \\
&\geq  \int_{\R^{N}\setminus \Omega_{n}} G(x, u_{n}) \, dx.
\end{align*}
But this gives a contradiction because of $G(x, u_{n}) \geq \delta$ a.e. $x\in \R^{N}\setminus \Omega_{n}$ and (\ref{(3.24)}). 
Now we assume that $a= \infty$ in $(H4)$ and $(J2)$ hold. Let $z_{n}$ be the sequence introduced in Lemma \ref{lemma3.6}. 

We claim to prove that 
\begin{equation}\label{(3.29)}
\lim_{n\rightarrow +\infty} I_{\lambda_{n}} (z_{n}) =+\infty. 
\end{equation} 
We recall that $||u_{n}||_{H^{s}(\R^{N})}\rightarrow +\infty$. Assume by contradiction that 
\begin{equation}\label{(3.30)}
\liminf_{n\rightarrow \infty} I_{\lambda_{n}}(z_{n}) \leq M < \infty. 
\end{equation}
Consider the following sequence 
$$
\xi_{n} = \sqrt{4M} \frac{u_{n}}{||u_{n}||_{H^{s}(\R^{N})}}= \sqrt{4M} w_{n} . 
$$
Then, $\xi_{n}$ is bounded in $H^{s}(\R^{N})$, $\xi_{n}$ vanishes and by Lemma \ref{fqt}  
\begin{equation}\label{convergence}
\xi_{n}\rightarrow 0 \mbox{ in } L^{q}(\R^{N}) , \mbox{ for any } q\in \Bigl(2,\frac{2N}{N-2s}\Bigr).  
\end{equation}
Thus, by (\ref{(3.2)}), Theorem \ref{ce} and (\ref{convergence}) we deduce
$$
\int_{\R^{N}} F(x, \xi_{n}) \, dx \rightarrow 0. 
$$
So we get, for $n\in \N$ large enough, 
\begin{equation*}
\label{(3.31)}
 I_{\lambda_{n}}(z_{n})  \geq I_{\lambda_{n}}(\xi_{n}) = 2M - \lambda_{n} \int_{\R^{N}} F(x, \xi_{n}) \, dx \geq M
\end{equation*} 
which is incompatible with (\ref{(3.30)}). Now, by using $I'_{\lambda_{n}}(z_{n})z_{n}=0$ for any $n\in \N$ and (\ref{(3.29)}), we obtain 
\begin{equation*}
\label{(3.32)}
\lambda_{n} \int_{\R^{N}} G(x, z_{n}) \, dx= I_{\lambda_{n}}(z_{n}) - \frac{1}{2} I'_{\lambda_{n}}(z_{n})z_{n} =  I_{\lambda_{n}}(z_{n})\rightarrow +\infty . 
\end{equation*}
But this is impossible because $(J2)$ and (\ref{(3.7)}) give  
\begin{equation*}
%\label{(3.33)}
\int_{\R^{N}} G(x, z_{n}) \, dx \leq D \int_{\R^{N}} G(x, u_{n}) \, dx \leq Dc_{1}. 
\end{equation*}

\end{proof}

\end{document}